\documentclass[a4paper]{article}
\usepackage[english]{babel}
\usepackage[utf8x]{inputenc}
\usepackage{amsmath}
\usepackage{graphicx}
\usepackage[colorinlistoftodos]{todonotes}
\newtheorem{theorem}{Theorem}

\newtheorem{corollary}[theorem]{Corollary}

\newtheorem{lemma}[theorem]{Lemma}

\newtheorem{remark}[theorem]{Remark}

\def\qed{\vbox{\hrule
 \hbox{\vrule\hbox to 5pt{\vbox to 8pt{\vfil}\hfil}\vrule}\hrule}}

\def\endproof{\unskip \nobreak \hskip0pt plus 1fill \qquad \qed \par}

\usepackage{tikz}
\usepackage{tkz-graph}

\begin{document}
\title{New  Bounds for the Signless Laplacian Spread}
\author{E. Andrade, Geir Dahl, Laura Leal, Mar\'ia Robbiano}

\maketitle

\section{Abstract}
Let $G$ be a simple graph. The signless Laplacian spread of $G$ is defined as the maximum distance of pairs of its
signless Laplacian eigenvalues. This paper establishes some new bounds, both lower and upper, for the signless Laplacian spread. Several of these bounds depend on invariant parameters of the graph. We also use a minmax principle to find several lower bounds for this spectral invariant.

\section{Introduction}

	In this paper we study an spectral invariant called signless Laplacian spread, defined as the difference between the maximum and minimum signless Laplacian eigenvalues. Some inequalities using the vertex bipartiteness value are explored and some relations  between the mentioned parameter and the signless Laplacian spread are studied.

 We deal with  a simple graph $G$ with
vertex set $\mathcal{V}(  G)  $ of cardinality $n$ and edge set
$\mathcal{E}(G)$ of cardinality $m$; we call this an $(n,m)$-graph.  An edge $e\in\mathcal{E}(G)$
with  end vertices $u$ and $v$ is denoted by $uv$, and we say that  $u$ and $v$ are neighbors.  A vertex $v$ is incident to an edge $e$  if $ v \in e $. $N_{G}(v)$ is the set of
neighbors of the vertex $v$, and its cardinality  is the degree of $v$,  denoted by $d(v)$.  Sometimes,
after a labeling of the vertices of $G$, a vertex $v_i$ is simply written $i$ and an edge
$v_{i}v_{j}$ is written $ij$, and we write $d_i$ for $d(v_i)$.  The
minimum and maximum vertex degree of $G$ are denoted by
$\delta(G)$ (or simply $\delta$) and $\Delta(G)\ $(or simply $\Delta$),
respectively. A $q$-\emph{regular} graph $G$ is a graph where every vertex has
degree $q$. $K_n$ is the complete graph of order $n$,  and $P_n$ (resp. $C_n$) is the  path (resp. cycle) with $n$ vertices.
A bipartite graph with bipartition $(X,Y)$ is denoted by $G(X,Y)$ (so any edge has one end vertex in $X$ and the other  in $Y$). $K_{p,q}$ is the complete bipartite graph with bipartition $\left(
X,Y\right)  $ with $\left\vert X\right\vert =p$,  $\left\vert
Y\right\vert =q$. A graph is called semi-regular bipartite if it is bipartite and the vertices belonging to the same part have equal degree.
We denote by $G \cup H$ the vertex disjoint union of graphs $G$ and $H$. A subgraph $H$ of $G$ is an \textit{induced subgraph} whenever two vertices of
$\mathcal{V}\left(  H\right)  $ are adjacent in $H$ if and only if they are adjacent in
$G$. If $H$ is an induced subgraph of $G$, the graph $G-H$ is the induced subgraph of $G$ whose vertex set is $\mathcal{V}\left( G-H\right) =\mathcal{V}(G) \setminus  \mathcal{V}(H).$
We only consider graphs without isolated vertices.
  Let $d_{1}, d_2, \ldots,d_{n}%
$\ be the  vertex degrees of $G$.  $A_{G}=\left(
a_{ij}\right)  $ denotes the adjacency matrix of $G$. The spectrum of
$A_{G}$ is called the \emph{spectrum} of $G$ and its elements  are called the
\emph{eigenvalues of} $G$. The vertex degree matrix $D_{G}$\ is the $n\times
n$\ diagonal matrix of the vertex degrees $d_i$ of $G$.  The \emph{signless Laplacian matrix} of $G$ (see
e.g. \cite{lapl4}) is defined by{\small \ }%
\begin{equation}
 \label{inertia1}
   Q_{G}=D_{G}+A_{G}.
\end{equation}
So, if $Q_G=(q_{ij})$, then $q_{ij}=1$ when $ij \in\mathcal{E}(G)$,  $q_{ii}=d_{i}$, and the remaining entries are zero. The
signless Laplacian matrix is  nonnegative and symmetric. The signless
Laplacian spectrum of $G$ is the spectrum of $Q_{G}$. Similarly, the matrix
\[
    L_{G}=D_{G}-A_{G}
\]
is the\textit{ }\emph{Laplacian matrix} of $G$ (\cite{lapl4,lapl1,lapl2}). For all these matrices we may omit the subscript $G$ if no misunderstanding should arise.

For a real symmetric matrix $W_{G}$, associated to a graph $G$, its
spectrum (the multiset of the eigenvalues of $W_{G}$) is
denoted by $\sigma_{W_{G}}$, and we let $\eta_{i}(  W_G)$ denote  the  $i$'th largest eigenvalue of $W_{G}$. The $i$'th largest eigenvalue of $A_{G}$ ($L_G$, $Q_G$, respectively) is denoted by $\lambda_i(G)$ ($\mu_i(G)$, $q_i(G)$, respectively). Sometimes they are simply denoted by
$\lambda_i$ ($\mu_i, q_i$, respectively).

For a graph $G$, its \textit{line graph} $\mathcal{L}_{G}$ is the graph with vertex set  $\mathcal{E}(  G)$ and where two edges in $\mathcal{E}(  G)$ are adjacent in $\mathcal{L}_{G}$ whenever the corresponding edges in $G$ have a common vertex.

Let
${I}_{G}$ be the \textit{vertex-edge incidence matrix} of the $(n,m)$-graph
$G$, defined as the $n\times m$ matrix whose $(i,j)$-entry is $1$ if vertex
$v_{i}$ is incident to the edge$\ e_{j}$, and $0$ otherwise. It is well known (see e.g.
\cite{BrualdiRyser91,lapl4}) that
\begin{equation}
 \label{sec}
 \begin{array}{ll}
   I_{G}I_{G}^{T} &  =D_{G}+A_{G}=Q_{G}, \\*[\smallskipamount]
   I_{G}^{T}I_{G} &  =2\,I_{m}+A_{\mathcal{L}_{G}},
 \end{array}
\end{equation}
where $I_{m}$ denotes the identity matrix of order $m$, so the matrices $Q_G$ and $2I_m+A_{{\mathcal{L}_{G}}}$ share the same nonzero eigenvalues.
As a consequence, if $G$\ is an $\left(  n,m\right)  $-graph, then $q_{i}(
G)  =2+\lambda_{i}\left(  \mathcal{L}_{G}\right)$ for $i=1, 2, \ldots, k$ where $k=\min\left\{  m,n\right\}$ and $\lambda_i (\mathcal{L}_{G})$ is the  $i$'th largest eigenvalue of $\mathcal{L}_{G}$. Moreover, if $m>n$, then
$\lambda_{i}\left(  \mathcal{L}_{G}\right)  =-2$, for $m\geq i\geq n+1$\ and
if $n>m$, then $q_{i}=0$\ for $n\geq i\geq m+1.$

An orientation of a graph $G$ is the directed graph obtained from $G$ by replacing every edge $uv$ by one of the pairs (arcs) $(u,v)$ or $(v,u)$. We let $\mathcal{O}(G)$ denote the set of all orientations of $G$.
 For $G^{\prime
}\in\mathcal{O}(G)$, its $(0, \pm1)$-\textit{incidence matrix}, denoted by $K_{G^{^{\prime}}}=\left(  \rho_{ij}\right)  $,$\ i=1,2, \ldots, n;\ j=1, 2, \ldots,m$, is given by%

\begin{equation}
\rho_{ij}=\left\{
\begin{array}{rlll}
   1 & \mathrm{if }\;\;  e_{j}= (v,v_{i})  &\text{for some }%
v\in\mathcal{V}\left(  G^{\prime}\right), \\
-1 & \mathrm{if }\;\;  e_{j}=(v_{i},v)  &\text{for some }%
v\in\mathcal{V}\left(  G^{\prime}\right), \\
0 &   \mathrm{otherwise.}%
\end{array}
\right.  \label{indi}%
\end{equation}
Then,  whatever the orientation of
the edges (\cite{inci,Merris}), the matrix $K_{G^{^{\prime}}}$ satisfies the following identity
\begin{equation}
K_{G^{^{\prime}}}K_{G^{^{\prime}}}^{T}=D_{G}-A_{G}=L_{G}. \label{impeq2}%
\end{equation}
By the relations (\ref{sec}) and (\ref{impeq2}) it is clear that the matrices
$Q_{G}$ and $L_{G}$ are both positive semidefinite. Moreover,  the all ones $n$-dimentional
vector $\mathbf{e}$ is an eigenvector of the Laplacian matrix for the eigenvalue $0$.

Note: We treat vectors in $\mathbf{R}^n$ as column vectors, but identify these with the corresponding $n$-tuples.
\section{The spread of symmetric matrices}
 \label{sec:spread}

This section collects some general results that are known for the spread of a symmetric matrix.

Let $\omega_{i}$ be the $i$'th largest eigenvalue of a symmetric matrix $W$. The
\emph{spread} of $W$ is defined by
\[
   s(W)  = \omega_1 -\omega_n.
\]
There are several papers devoted to this parameter, see for instance
\cite{Jiang-Zhan,Jonhson et al, Mirsky, Nylen}.
For a real rectangular matrix  $W=(w_{ij})$, let  $\left\Vert W\right\Vert _{F}=(\sum_{ij} w_{ij}^2)^{1/2}$
be the Frobenius  norm of $W$. When $W$ is square its trace will be denoted by ${\rm tr}\, W$.
In 1956, Mirsky proved the following inequality.

\begin{theorem}
\label{MK1}{\rm (\cite{Mirsky})}
Let $W$ be an $n\times n$  normal matrix. Then
\begin{equation}
 \label{M1}
   s(W)  \leq \left(2\left\Vert W\right\Vert _{F}^{2}-\frac{2}%
{n}\left(  {\rm tr }\, W\right)  ^{2}\right)^{1/2}
\end{equation}
with equality if and only if the  eigenvalues $\omega_{1}, \omega_2, \ldots,\omega_{n}$ of $W$ satisfy the following condition
%
\[
  \omega_{2}=\omega_3=\cdots=\omega_{n-1}=\frac{\omega_{1}+\omega_{n}}{2}.
\]

\end{theorem}

Among the results obtained for the spread of a symmetric matrix $W=\left(
w_{ij}\right)  $ we mention the following lower bound obtained in
\cite{Barnes-Hoffman}.

\begin{theorem}
\label{barnes}{\rm (\cite{Barnes-Hoffman})} Let $W=\left(  w_{ij}\right)  $ be an
$n\times n$ normal and symmetric matrix. Then
\begin{equation}
 \label{relationspread}
 s(W)  \geq\max_{i,j}\left(  \left(  w_{ii}-w_{jj}\right)
^{2}+2\sum_{s\neq j}\left\vert w_{js}\right\vert ^{2}+2\sum_{s\neq
i}\left\vert w_{is}\right\vert ^{2}\right)  ^{1/2}.
\end{equation}
\end{theorem}

Some other  lower bounds for the spread
of Hermitian matrices are found in \cite{Jiang-Zhan}, and in some cases these improve the lower bound in
(\ref{relationspread}).

\begin{theorem}
 \label{J-Z1 copy(2)}{\rm (\cite{Jiang-Zhan})}
 For any Hermitian matrix $W=\left(
w_{ij}\right)$
\[
 s(W)^{2}\geq\max_{i\neq j}\left\{  \left(  w_{ii}%
-w_{jj}\right)  ^{2}+2%
{\displaystyle\sum\limits_{k\neq i}}
\left\vert w_{ik}\right\vert ^{2}+2%
{\displaystyle\sum\limits_{k\neq j}}
\left\vert w_{jk}\right\vert ^{2}+4e_{ij}\right\}  \text{,}%
\]
where $e_{ij}=2f_{ij}$ if $w_{ii}=w_{jj}$ and otherwise
\[
 e_{ij}= \min\left\{  \left(  w_{ii}-w_{jj}\right)  ^{2}+2\left\vert \left(
w_{ii}-w_{jj}\right)  ^{2}-f_{ij}\right\vert ,\frac{f_{ij}^{2}}{\left(
w_{ii}-w_{jj}\right)  ^{2}}\right\}
\]
with
\[
f_{ij}=\left\vert
{\displaystyle\sum\limits_{k\neq i}}
\left\vert w_{ik}\right\vert ^{2}-%
{\displaystyle\sum\limits_{k\neq j}}
\left\vert w_{jk}\right\vert ^{2}\right\vert .
\]
\end{theorem}

\section{Spreads associated with graphs}

Let $G$ be an $(n,m)$-graph. We now consider different notions of spread  based on  matrices associated with $G$.

As before $A_G$ is the adjacency matrix of $G$ and we consider
\[
    s(G)=s(A_G)
\]
which is called the {\em spread} of $G$ (\cite{G}).
Let $\mathbf{\mu}(G)  \mathbf{=}\left(  \mu_{1}, \mu_2, \ldots,\mu
_{n}\right)  $ be the vector whose components are the Laplacian eigenvalues of $G$ (ordered decreasingly, as usual). The {\em Laplacian spread}, denoted by
$s_{L}(G)$,  is defined (\cite{Zheng et al}) by
\[
    s_{L}(G)  =\mu_{1}-\mu_{n-1}.
\]
Note that $\mu_n=0$.
Let $q(G)=(q_{1}, q_2, \ldots,q_{n})$ be the vector whose components are the signless Laplacian eigenvalues ordered decreasingly. The {\em signless Laplacian spread}, denoted by
$s_{Q}(G)$, is defined (\cite{Liu2}, \cite{Carla}) as
\[
     s_{Q}(G)  =q_{1}-q_{n}.
\]
\begin{remark}
\label{rem2}
{\rm
Some basic properties of these notions are as follows:

\begin{itemize}

\item[(i)] Let $G$ be a graph of order $n$ with largest vertex degree
$\Delta$. From  Theorem \ref{barnes} one can easily see that  $s(G)
\geq 2\sqrt{\Delta}$. Moreover, if $G=K_{1,n-1}$,  equality holds.

\item[(ii)]
If $G$ is a regular graph, then $s_{Q}(G)  =s(G),$ (\cite{Liu2}).

\item[(iii)]
From the relation $Q_{G}=2A_{G}+L_{G}$ it follows that  $q_{1}\geq2\lambda_{1}$ as
$L_{G}$ is positive semidefinite
(and it is known that  equality holds if an only if $G$ is a regular graph), (see e.g \cite{Das2,XDZ2}). Moreover,  as $\lambda_1$ is the spectral radius of $A_G$,  $2\lambda_{1}\geq\lambda_{1}%
-\lambda_{n}=s(G) $ with equality if and only if $G$ is a
bipartite graph. Therefore
\[
  s(G) \leq q_{1}%
\]
with equality if and only if $G$ is a regular, bipartite graph.

\item[(iv)] We recall the Weyl's inequalities for a particular case in what follows.
Consider  two $n\times n$ Hermitian matrices $W$ and $U$ with
eigenvalues (ordered nonincreasingly) $\omega_{1}, \omega_2, \ldots,\omega_{n}$
and $x_{1}, x_2, \ldots,x_{n}$, respectively, and the Hermitian matrix $T=W+U$ with
 eigenvalues $\tau_{1}, \tau_2, \ldots,\tau_{n}$ (ordered nonincreasingly). Then
 the following inequalities hold%
\[
   \omega_{n}+x_{i}\leq\tau_{i}\leq\omega_{1}+x_{i} \;\;\;\;(i \le n).
\]
Thus $\omega_{n}+x_{1}\leq\tau_{1}\leq\omega_{1}+x_{1}$ and $\omega_{n}%
+x_{n}\leq\tau_{n}\leq\omega_{1}+x_{n}$. Therefore
\[
x_{1}-x_{n}+\omega_{n}-\omega_{1}\leq\tau_{1}-\tau_{n}\leq\omega_{1}%
-\omega_{n}+x_{1}-x_{n}%
\]
which gives the following inequalities for the spread of these matrices
\[
 \left \vert s(U)  -s(W)  \right\vert \leq s(T)  \leq s(W)  +s(U).
\]

\item[(v)]
Let $G$ be a graph with smallest and largest vertex degree $\delta$ and
$\Delta$, respectively. By the previous item, as $Q_G=D_G+A_G$,
\[
\left\vert \Delta-\delta-s(G) \right\vert \leq s_{Q}(G)  \leq s(G) +\Delta-\delta.
\]
\end{itemize} \endproof
}
\end{remark}

The next inequality establishes a relation between the largest Laplacian eigenvalue and the largest signless Laplacian eigenvalue.
\begin{lemma}
{\rm (\cite{XDZ})}
\label{research}Let $G$ be a graph. Then
\[
\mu_{1}(  G)  \leq q_{1}(  G)  .
\]
Moreover if $G$ is connected, then the equality holds if and only if $G$ is a
bipartite graph.
\end{lemma}

The second smallest Laplacian eigenvalue of a graph $G$ is  known as
the algebraic connectivity (\cite{Fiedler}) of $G$ and denoted by $\mathit{a}(G)
$. If $G$ is a
non-complete graph, then \textit{a}$\left(  G\right)  \leq\kappa_{0}(G)$,
where $\kappa_{0}(G)$ is the vertex connectivity of $G$ (that is, the
minimum number of vertices whose removal yields a disconnected graph). Since
$\kappa_{0}(G)  \leq\delta(G)$, it follows that \textit{a}%
$\left(  G\right)  \leq\delta.$ The graphs for which the algebraic
connectivity attains the vertex connectivity are characterized in
\cite{kirkland_et_al_02}. One also has    (\cite{Enide15, Liu})
\begin{equation}
s_{L}(G)  \geq\Delta(G)+1-\delta\left(  G\right)  .
\label{basic_lower_bound}%
\end{equation}
For a survey on algebraic connectivity, see \cite{Nair}. Moreover, it is worth to conclude that the result in
(\ref{basic_lower_bound}) together with the result in Remark  \ref{rem2} (v) imply that
$$s_{Q}(G)\leq s(G)+s_{L}(G)-1.$$

\begin{remark}
 \label{rem1}
 {\rm
If $G$ is a connected $(n,m)$-graph such that $m\leq n-1$, then $G$ does not have cycles and thus,
it is bipartite. Therefore $s_{Q}\left(  G\right)  =q_{1}=\mu_{1}.$ As in the literature there are many known lower and upper bounds for this eigenvalue,  from now on we only treat the case $m\geq n$.
} \endproof
\end{remark}

Some results on  $s_{Q}(G)$ can be found in \cite{Carla,Liu}. Some of them are listed below.

\begin{theorem}
{\rm (\cite{Carla,Yi-Zhenf and Fallat})}  For any  graph $G$ with $n$ vertices
\[
    s_{Q}(P_{n})=2+2\cos(\pi/n) \leq s_{Q}(G)
\]
with equality if and only if $G=P_{n}$ or $G=C_{n}$ in case $n$ odd.
\end{theorem}

\begin{theorem}
{\rm (\cite{Carla})}
 For any  graph $G$ with $n \ge 5$ vertices
\[
   s_{Q}(G)\leq2n-4,
\]
and equality holds if and only if $G=K_{n-1}\cup K_{1}.$
\end{theorem}

\begin{theorem}
{\rm (\cite{Liu2})} If $G$ is a connected graph, then
\[
\Delta(G)+1-\delta(G)  < s_{Q}(G) \leq \max\left\{d(v)  +\frac{1}{d(v) }\sum_{uv\in\mathcal{E}(G)  }d(u): v\in\mathcal{V}\left(  G\right)  \right\},
\]
where the upper bound holds with equality if and only if $G$ is regular or semi-regular
bipartite.
\end{theorem}

%
%

%
The minimum number of vertices (resp., edges) whose deletion yields a
bipartite graph from $G$ is called the \textit{vertex bipartiteness} (resp.,
\textit{edge bipartiteness}) of $G$ and it is denoted $\upsilon_{b}\left(
G\right)  $ (resp., $\epsilon_{b}(G)$), see \cite{Fan-Fallat}.
Let $q_{n}$ be the smallest eigenvalue of $Q_{G}$. In \cite{Fan-Fallat}, one established the inequalities
\begin{equation}
 \label{ineq1}%
   q_{n}\leq\upsilon_{b}(G)\leq\epsilon_{b}(G).
\end{equation}


%
In \cite{Yi-Zhenf and Fallat} some important relationships between
$\epsilon_{b}(G)$ and $s_{Q}(G)$ were found, and  it was shown that if
%
\[
   s_{Q}(G)  \geq 4
\]
with equality if and only if $G$ is one of the following graphs:
$K_{1,3},\ K_{4},\ $two triangles connected by an edge, and $C_{n}$ with $n$ even.

\section{Lower bounds}
 \label{sec:lower}

We now present some new lower bounds for the signless Laplacian
spread. The first results involve the vertex bipartiteness graph invariant.

\begin{theorem}\label{NEW_1}
Let $G$ be a connected graph with $n$ vertices and vertex bipartiteness
$\upsilon_{b}(G)$. Then
\begin{equation}
s_{Q}(  G)  \geq\mu_{1}(  G)  -\upsilon_{b}%
(G).\label{lbb}%
\end{equation}
If $G$ is a connected bipartite graph, then  equality holds in $(\ref{lbb})$.
\end{theorem}

\begin{proof}
By Lemma \ref{research}
\[
q_{1}(  G)  \geq\mu_{1}(  G)  .
\]
By using (\ref{ineq1}) the inequality in the statement follows. If $G$ is a
bipartite graph, then  $\upsilon_{b}(G)=0=q_{n}$ and $q_{1}(  G)
=\mu_{1}(  G)  $ then the equality follows.
\end{proof}

\begin{corollary}
Let $G$ be a connected $(n,m)$-graph and vertex
bipartiteness $\upsilon_{b}(G)$.  Then
\begin{equation}
s_{Q}(  G)  \geq\frac{4m}{n}-\upsilon_{b}(G).\label{lbbm}%
\end{equation}
Equality holds here if $G$ is a regular bipartite graph.
\end{corollary}

\begin{proof}
Let $\mathbf{e=}\left(  1,\ldots,1\right)  $, the all ones vector, then
$q_{1}(  G)  \geq\frac{\mathbf{e}^{T}Q_G  \mathbf{e}%
}{\mathbf{e}^{T}\mathbf{e}}=\frac{4m}{n}\ $\ with equality if $G$ is a regular
graph. Again, using (\ref{ineq1}) the inequality in the statement follows.
If $G$ is a regular bipartite graph then $\upsilon_{b}(G)=0$ and
$\ q_{1}(  G)  =\frac{4m}{n}\ $thus, the equality in (\ref{lbbm})
follows.
\end{proof}



\medskip
\begin{theorem} \label{signvertex}
Let $G$ be a connected graph with $n$ vertices and vertex bipartiteness
$\upsilon_{b}(G)$. Then
\[
s_{Q}(  G)  \geq 2\lambda_{1}(  G)  -\upsilon_{b}(G).
\]
Equality holds if $G$ is a regular, bipartite graph.
\end{theorem}
\begin{proof}
Since
\[
q_{1}(  G)  \geq2\lambda_{1}(  G),
\] using (\ref{ineq1}), the inequality in the statement follows. Moreover, if $G$ is regular then $q_{1}\left(
G\right)  =2\lambda_{1}(  G), $ (see e.g. \cite{ Das2,XDZ2}). Additionally, if $G$ is bipartite, then $\upsilon_{b}(G)=0=q_{n}$. Therefore the equality follows.
\end{proof}

\medskip
The bounds in the previous theorems and corollaries show connections between signless Laplacian spread and the vertex bipartiteness $v_b(G)$. It is therefore natural to ask how difficult this parameter is to compute.
The {\em vertex bipartization} problem is to find the minimum number of vertices in a graph whose deletion leaves a subgraph which is bipartite. This problem is NP-hard, even when restricted to graphs of maximum degree 3, see  \cite{Choi}. Actually, this problem has several applications, such as in via minimization in the design of integrated circuits (\cite{Choi}). Exact algorithms and complexity of different variants have been studied, see \cite{Raman}. For the parameterized version, where $k$ is fixed and one asks for $k$ vertices whose deletion leaves a bipartite subgraph, an algorithm of complexity $O(3^k\cdot kmn)$ was  found in \cite{Reed}. Similarly, it is NP-hard to compute the edge bipartiteness $\epsilon_b(G)$, even if all degrees are 3, see  (\cite{Choi}).
Finally, more general vertex and edge deleting problems were studied in \cite{Yannakakis}, and NP-completeness of a large class of such problems was shown.

\medskip

Recall that an induced subgraph is determined by its vertex set. In fact,
deleting some vertices of $G$ together with the edges incident on those
vertices we obtain an induced subgraph.  A set of vertices that induces an empty subgraph is called an
\textit{independent set}. The number of vertices in a maximum independent set of $G$ is
called the \textit{independence number} of $G$ and it is denoted by
$\alpha\left(  G\right) $. The problem of finding the independence number of a graph $G$ is also NP-hard, see \cite{Garey, Karp}, whereas the spectral bounds can be determined in polynomial time.

\begin{lemma}\label{new101}
Let $G$ be a graph with $n$ vertices and independence number
$\alpha\left(  G\right)  $. Then
\begin{equation*}
q_n \leq \upsilon_{b}(  G)  \leq\epsilon_{b}(  G)  \leq
\frac{\left(  n-\alpha (  G)  \right)  \left(  n-\alpha (
G)  -1\right)  }{2}.
\end{equation*}
\end{lemma}

\begin{proof} Let $S\subseteq \mathcal{V}\left(  G\right) $ be  an independent set of vertices with
cardinality $\alpha=\alpha\left(  G\right)  $ and $H$ be an induced subgraph
of $G$ such that $\mathcal V \left(  H\right)  =\mathcal V \left(  G\right)  \backslash S$. The adjacency matrix of $G$ becomes
\begin{equation*}
A_{G}=%
\begin{pmatrix}
\mathbf{0} & C\\
C^{T} & A_{H}%
\end{pmatrix}
\end{equation*}
where $\mathbf{0}$ is the square zero matrix of order $\alpha$.  Note that the
cardinality of the set of edges of $H$ satisfies
\[
\left\vert \mathcal{E}(  H)  \right\vert \leq\frac{(  n-\alpha)
\left(  n-\alpha-1\right)  }{2}.
\]
The result is obtained since deleting all the edges of $H$ yields a bipartite graph from $G$.
\end{proof}

\begin{corollary}

Let $G$ be an $\left(  n,m\right)  $-graph with independence number
$\alpha=\alpha\left(  G\right)  $.\ If
\begin{equation}
 n\left(  n-\alpha\right)  \left(  n-\alpha-1\right) \le 8m, \label{cond}%
\end{equation}
\ then%
\[
s_{Q}\left(  G\right)  \geq2\lambda_{1}-\upsilon_{b}\left(  G\right)
\geq\frac{4m}{n}-\upsilon_{b}\left(  G\right)  \geq0.
\]
\end{corollary}

\begin{proof}
Let $\mathbf{e}=\left(  1,1, \ldots,1\right)  $ be the all ones vector, then
\[
q_{1}\geq\frac{\mathbf{e}^{T}\left(  Q_{G}\right)  \mathbf{e}}{\mathbf{e}%
^{T}\mathbf{e}}=\frac{4m}{n}%
\]
with equality if $G$ is a regular graph. On the other hand, the condition in
(\ref{cond}) implies that
\[
\frac{4m}{n}-\frac{\left(  n-\alpha  \right)  \left(
n-\alpha  -1\right)  }{2}\geq0.
\]
$\ $Thus, by Lemma \ref{new101}
\[
\frac{4m}{n}-\upsilon_{b}\left(  G\right)  \geq\frac{4m}{n}-\frac{(n-\alpha) (  n-\alpha -1)  }{2}\geq0.
\]
By Theorem \ref{signvertex} the first inequality is obtained. The second inequality follows from the fact that
\[
2\lambda_{1}\geq2\left(  \frac{\mathbf{e}^{T}\left(  A_{G}\right)  \mathbf{e}%
}{\mathbf{e}^{T}\mathbf{e}}\right)  =\frac{4m}{n}.
\]
\end{proof}

\begin{remark}
\label{Rineq}
{\rm
Note that if $\alpha\left(  G\right)  =n-k$ and as $m\leq
\frac{n\left(  n-1\right)  }{2}$ a necessary condition for $(\ref{cond})$ is
\[
4\left(  n-1\right)  \geq k\left(  k-1\right).
\]  \endproof
}
\end{remark}

\medskip
Recall the identity
\[
\alpha(G)+ \tau(G) =n,
\]
where $\tau(G)$ is the vertex cover number of $G$  (that is the size of a minimum vertex cover in a graph $G$).
Finding a minimum vertex cover of a general graph is an NP-hard problem however, for the  bipartite graphs, the vertex cover number is equal to the matching number. Therefore, from the previous remark
a necessary condition for  (\ref{cond}) is, in this case, $4(n-1)\geq  \tau(G) (\tau(G)-1).$

Now, using Theorems \ref{barnes}  and \ref{J-Z1 copy(2)} we derive the following results.
\begin{theorem}
 \label{thm:two-lower-bounds}
Let $G$ be a graph of order $n$ with largest vertex degree $\Delta$ and
smallest vertex degree $\delta$. If $\Delta-\delta\geq2$, then
\[
   s_{Q}(  G)  \geq \left(  \left(
\Delta-\delta\right)  ^{2}+2\Delta+2\delta\right)  ^{1/2}.
\]
and otherwise $($when $\Delta-\delta \le 1$$)$
\[
   s_{Q}(  G)  \geq   2\sqrt{\Delta}.
\]
Equality holds for $G\simeq$ $K_{2}.$
\end{theorem}

\begin{proof}
Let $Q_G  =\left(  q_{ij}\right)  $ be the signless Laplacian
matrix of $G$, then $Q_G$ is an $n\times n$ normal matrix and
by Theorem \ref{barnes}
\[
   s_{Q}(G)=s(Q_G)  \geq \Upsilon
\]
where
\[
 \begin{array}{ll}
 \Upsilon &=\max_{i,j}\left(  \left(  q_{jj}-q_{ii}\right)
^{2}+2\sum_{s\neq j}\left\vert q_{js}\right\vert ^{2}+2\sum_{s\neq
i}\left\vert q_{is}\right\vert ^{2}\right)  ^{1/2} \\*[\smallskipamount]
     &= \max_{i,j}\left(  \left(  d_j-d_i\right)
^{2} + 2(d_j+d_i)\right)^{1/2}.
 \end{array}
\]
In this maximization  we may assume (by symmetry)  that $d_j \ge d_i$. Moreover,  by fixing $d_j-d_i$ to some number $k \in \{0, 1, \ldots, \Delta-\delta\}$, we get
\[
 \begin{array}{ll}
 \Upsilon &=\max_k \max_{d_j-d_i=k} \left(  \left(  d_j-d_i\right)
^{2} + 2(d_j+d_i)\right)^{1/2} \\*[\smallskipamount]
  &= \max_k \max_{d_j-d_i=k} \left(  k^2 + 2(2d_i+k)\right)^{1/2} \\*[\smallskipamount]
  &= \max_k  \left(  k^2 + 2(2(\Delta-k)+k)\right)^{1/2}
 \end{array}
\]
as $k^2 + 2(2d_i+k)$ is increasing in $d_i$. So
$\Upsilon =\max_k  \left(  k^2 + 4\Delta-2k \right)^{1/2}$. But  $k^2 + 4\Delta-2k $ is a convex quadratic polynomial in $k$ so its maximum over $k \in \{0, 1, \ldots, \Delta-\delta\}$ occurs in one of the two endpoints. Therefore
\[
   \Upsilon =\max \{2\sqrt{\Delta}, \left( (\Delta-\delta)^2 +2(\Delta+\delta) \right)^{1/2}\}
\]
which gives the desired result.
\end{proof}

\medskip
Let $\mathcal{V}\left(  \Delta\right)  =\left\{  v\in\mathcal{V}\left(
G\right)  :d\left(  v\right)  =\Delta\right\}  $ and $\mathcal{V}\left(
\delta\right)  =\left\{  v\in\mathcal{V}\left(  G\right)  :d\left(  v\right)
=\delta\right\}  $.

\begin{theorem}
\label{barnes_3}Let $G$ be a graph of order $n$ with largest vertex degree
$\Delta$ and smallest vertex degree $\delta$.
\[
s_{Q}(  G)  \geq\left(  \left(  \Delta-\delta\right)  ^{2}%
+2\Delta+2\delta+4\right)  ^{\frac{1}{2}}.
\]
We have equality for $G\simeq K_{2}$ and $G\simeq K_{1,3}.$
\end{theorem}

\begin{proof}
Let $Q(  G)  =\left(  q_{ij}\right)  $ be the signless Laplacian
matrix of $G$, then $Q_G$ is an $n\times n$ symmetric matrix
and by Theorem \ref{J-Z1 copy(2)} we derive
\[
s_{Q}(G)=s(  Q_G )  \geq  \Gamma
\]
where
\[
\Gamma=\max_{i\neq j}\left(
\left(  q_{ii}-q_{jj}\right)  ^{2}+2\sum_{s\neq j}\left\vert q_{js}\right\vert
^{2}+2\sum_{s\neq i}\left\vert q_{is}\right\vert ^{2}+4e_{ij}\right)  ^{1/2},
\]
and $e_{ij}$ and $f_{ij}$ are given in Theorem \ref{J-Z1 copy(2)}.

Let $v_{i_{0}}\in \mathcal{V} (\Delta)  $ and $v_{j_{0}}%
\in\mathcal{V}\left(  \delta\right)  $.
If $q_{j_{0}j_{0}}=q_{i_{0}i_{0}}$, then  $e_{i_{0}j_{0}}=2f_{i_{0}j_{0}}$;  otherwise
\[
   e_{i_{0}j_{0}}=\min\left\{  \left(  q_{i_{0}i_{0}}-q_{j_{0}j_{0}}\right)  ^{2}+2\left\vert
\left(  q_{i_{0}i_{0}}-q_{j_{0}j_{0}}\right)  ^{2}-f_{i_{0}j_{0}}\right\vert
,\frac{f_{i_{0}j_{0}}^{2}}{\left(  q_{i_{0}i_{0}}-q_{j_{0}j_{0}}\right)  ^{2}%
}\right\}
\]
with
\[
f_{i_{0}j_{0}}=\left\vert
{\displaystyle\sum\limits_{k\neq i_{0}}}
\left\vert q_{i_{0}k}\right\vert ^{2}-%
{\displaystyle\sum\limits_{k\neq j_{0}}}
\left\vert q_{j_{0}k}\right\vert ^{2}\right\vert =\left\vert d\left(
v_{i_{0}}\right)  -d\left(  v_{j_{0}}\right)  \right\vert =\Delta-\delta.
\]
Therefore,
\[
e_{i_{0}j_{0}}=\min\left\{  \left(  \Delta-\delta\right)  ^{2}+2\left\vert
\left(  \Delta-\delta\right)  ^{2}-\left(  \Delta-\delta\right)  \right\vert
,1\right\}  =1.
\]
Thus $\Gamma \geq\left(  \left(  \Delta-\delta\right)  ^{2}+2\Delta
+2\delta+4\right)  ^{\frac{1}{2}}$ and the result follows.
\end{proof}

The next Corollary is a direct consequence of the previous theorem.

\begin{corollary} \label{3.9}
Let $G$ be a $k$-regular graph. Then
$s(G)=s_{Q}(G) \geq 2\sqrt{ k+1}.$
\end{corollary}

\medskip
Let $G$ be a graph with vertex degrees $d_{1}, d_2, \ldots,d_{n}$. Let
\[
M_{1}(  G)  =%
{\displaystyle\sum\limits_{i=1}^{n}}
d_{i}^{2},
\]
 be the first Zagreb index \cite{GD}.
In \cite{Pecaric} the following  inequality related to the Cauchy-Schwarz inequality is shown. It follows directly from the Lagrange identity (see \cite{Steele} concerning Lagrange identity and related inequalities).

\begin{lemma} {\rm \cite{Pecaric}}
 \label{Ozeki copy(1)}
Let $\mathbf{a}=\left(  a_{1}, a_2, \ldots,a_{n}\right)  $
and $\mathbf{b}=\left(  b_{1}, b_2, \ldots,b_{n}\right)  $ be two vectors with
$0<m_{1}\leq a_{i}\leq M_{1}$ and $0<m_{2}\leq b_{i}\leq
M_{2}$, for $i=1, 2, \ldots,n,$ for some constants $m_{1},m_{2},M_{1}%
\ $and$\ M_{2}$. Then
\begin{equation}
\left(
{\textstyle\sum\limits_{i=1}^{n}}
a_{i}^{2}\right)  \left(
{\textstyle\sum\limits_{i=1}^{n}}
b_{i}^{2}\right)  -\left(
{\textstyle\sum\limits_{i=1}^{n}}
a_{i}b_{i}\right)  ^{2}\leq\frac{n^{2}}{4}\left(  M_{1}M_{2}-m_{1}%
m_{2}\right)  ^{2}. \label{numerique}%
\end{equation}

\end{lemma}

By using the above result in what follows, we will obtain a lower bound for the
$s_{Q}(  G)  $ in terms of $M_{1}(  G)  $, $n$ and $m.$

\begin{theorem} \label{3.11}
\label{O1}Let $G$ be a connected graph with $n\geq2$ vertices. Then%
\[
   s_{Q}(  G)  \geq\frac{2}{n}\sqrt{nM_{1}(  G)
-4m^{2}+2mn}.
\]

\end{theorem}

\begin{proof}
In this proof we use Lemma \ref{Ozeki copy(1)} with $a_{i}=1$ and
$b_{i}=q_{i}$, for $1\leq i\leq n$. Since $0<1\leq a_{i}\leq1$, and
$0<q_{n}\leq b_{i}\leq q_{1},1\leq i\leq n.\ $\ Thus $M_{1}M_{2}=1q_{1}$ and
$m_{1}m_{2}=1q_{n}$.\ By Lemma \ref{Ozeki copy(1)}%
\[
\sum_{i=1}^{n}1\sum_{i=1}^{n}q_{i}^{2}-\left(  \sum_{i=1}^{n}q_{i}\right)
^{2}\leq\frac{1}{4}n^{2}\left(  q_{1}-q_{n}\right)  ^{2}%
\]
then%
\[
n\left(  2m+M_{1}\left(  G\right)  \right)  -4m^{2}\leq\frac{1}{4}n^{2}\left(
q_{1}-q_{n}\right)  ^{2}.%
\]
This gives
\[
\frac{8m+4M_{1}\left(  G\right)  }{n}-\frac{16m^{2}}{n^{2}}\leq s_{Q}%
^{2}(  G)
\]
and
\[
s_{Q}(  G)  \geq2\sqrt{\frac{nM_{1}(  G)  -4m^{2}%
+2mn}{n^{2}}}.
\]
\end{proof}

\begin{remark}
{\rm
Note that for a $k$-regular graph $G$ the lower bound given by Theorem \ref{3.11} is $ 2\sqrt{k}$. So, for regular graphs,  it is worse than the other one given by Corollary \ref{3.9}.
} \endproof
\end{remark}

\section{Lower bounds based on a minmax principle}
 \label{sec:minmax}

In this section we introduce a  principle for  finding several  lower bounds for the signless Laplacian spread of a graph.

Let $B_n$ denote the unit ball in $R^n$. The next theorem gives a lower bound on the spread of a real symmetric matrix $A$. The result is actually known in a slightly different form (see below), but we give a new proof of this inequality, using ideas from minmax theory.

\begin{theorem}
  \label{thm:lbd}
  Let $A$ be a real symmetric matrix of order $n$. Then
  \begin{equation}
    \label{eq:minmax-bd}
    s(A) \ge 2 \|A\mathbf{x}-(\mathbf{x}^TA\mathbf{x})\mathbf{x}\| \;\;\;\;\mbox{\rm for all $\mathbf{x} \in B_n$.}
  \end{equation}
\end{theorem}
\begin{proof}
 Lemma 1 in \cite{Jiang-Zhan}  says that $s(A)=2\min _{ t\in R }{ \left| A-tI \right|  } $ where the minimum is over all $t \in R$ (this follows easily from the spectral theorem). Therefore
 %
\begin{equation}
 \label{eq:bound-inq}
\begin{array}{ll}
  (1/2)s(A)=\min _{ t\in R }{ \left\| A-tI \right\|  } \\ =\min _{ t\in R }{ \max _{ x\in B_{ n } }{ \left\| \left( A-tI \right) x \right\|  }  } \\ \geq \max _{ x\in B_{ n } }{ \min _{ t\in R }{ \left\| Ax-tx \right\|  }  } \\ =\max _{ x\in B_{ n } }{ \left\| Ax-(x^ { T}Ax)x) \right\|  } 
\end{array}
\end{equation}
The inequality above follows from standard minmax-arguments. In fact, for any function $f=f(\mathbf{x},t)$ defined on sets $X$ and $T$, we clearly have
$\inf_{t'\in T} f(\mathbf{x},t') \le f(\mathbf{x},t) \le \sup_{\mathbf{x}' \in X} f(\mathbf{x}',t)$ for all $\mathbf{x} \in X$ and $t \in T$. The desired inequality is then obtained by taking the infimum over $t$ in the last inequality, and then, finally, the supremum over $\mathbf{x}$. The final equality in (\ref{eq:bound-inq})  follows as this is a least-squares problem in one variable $t$, for given $\mathbf{x} \in B_n$, so geometrically $t$ is chosen so that $t\mathbf{x}$ is the orthogonal projection of $A\mathbf{x}$ onto the line spanned by $\mathbf{x}$. The desired result now follows from (\ref{eq:bound-inq}).
\end{proof}

\medskip
Below we  rewrite the bound in the previous theorem. First, however, note from the proof that the bound in (\ref{eq:bound-inq}) expresses the following: {\em for any unit vector $\mathbf{x}$,  twice the distance from $A\mathbf{x}$ to the line spanned by $\mathbf{x}$ is a lower bound on the spread}. Thus, the bound has a simple geometrical interpretation. This may  be useful, in specific situations, in order to find an  $\mathbf{x}$ which gives a good lower bound.

Now, a straightforward computation shows that
\[
 \|A\mathbf{x}- (\mathbf{x}^{T}A\mathbf{x}) \mathbf{x} \|^{2}   = \mathbf{x}^{T}A^{2}\mathbf{x} -(\mathbf{x}^{T}A\mathbf{x})^{2}
\]
so  Theorem \ref{thm:lbd} says that
\begin{equation}
 \label{red}
   s(A) \ge 2\max\limits_{\mathbf{x}\in B_{n}}\sqrt{ \mathbf{x}^{T}
A^{2}\mathbf{x} -\left( \mathbf{x}^{T}A\mathbf{x}\right)  ^{2}}.
\end{equation}

Therefore this result is actually the result presented in \cite [Theorem 4]{MerikoskiKumar}. In \cite{MerikoskiKumar} the authors state that this result, in fact, goes back to Bloomfield and Watson in 1975, \cite[(5.3)]{BloomfieldWatson}, and it was rediscovered by Styan \cite[Theorem 1]{Styan}. See also \cite[section 5.4]{Jensen} and \cite{Jia}.

The result in Theorem \ref{thm:lbd} may also be reformulated  in terms of a nonzero vector $\mathbf{y}=(y_1, y_2, \ldots, y_n)$. Then $\mathbf{x}= (1/\|\mathbf{y}\|) \mathbf{y}$ is a unit vector, and a simple  calculation, using (\ref{red}),  gives
\begin{equation}
 \label{eq:lbd_reform}
   s(A) \ge 2
   \frac{\left( \sum_i y_{i}^{2} \sum_i
\tau_{i}^{2}  -\left(\sum_i y_{i}\tau_{i}\right)^{2}\right)^{1/2}}{\sum_i
y_{i}^{2}}
\end{equation}
where
\[
   \mathbf{\tau}=A\mathbf{y}=(\tau_{1}, \tau_2, \ldots,\tau_{n}).
\]

\begin{remark}
{\rm 
From the equality case of  Cauchy-Schwarz Theorem, the bound in  (\ref{eq:lbd_reform}) is equal to zero  when the vector $\tau$ and $\mathbf{y}$ are a linear combination of the vector $ \mathbf{e} $.
} \endproof
\end{remark}


\medskip
We may now obtain different lower bounds on the signless Laplacian spread $s_Q(G)$, for a graph $G$, by applying Theorem   \ref{thm:lbd}  to the signless Laplacian matrix $Q_G$ and choosing some specific unit vector $\mathbf{x}$, or a nonzero vector $\mathbf{y}$, and use (\ref{eq:lbd_reform}).

For instance, consider the simple choice $\mathbf{x}=\mathbf{e}_i$, the $i$th coordinate vector. Then $Q_G\mathbf{e}_i-(\mathbf{e}_i^TQ_G\mathbf{e}_i)\mathbf{e}_i =Q^{(i)}-d_i\mathbf{e}_i$ (where $Q_G^{(i)}$ is the $i$th column of $Q_G$), so $\eta(G,\mathbf{e}_i)= 2\sqrt{d_i}$. This gives
\[
  s_{Q}=  s(G) \ge 2\max_i \sqrt{d_i} = 2\sqrt{\Delta}
\]
which gives a short proof of the second bound (when $\Delta-\delta \le 1$) in Theorem \ref{thm:two-lower-bounds}. Another application of this principle is obtained by using $\mathbf{x}$ as the normalized all ones vector, which gives the following lower bound.

\begin{corollary}
  \label{cor:lbd}
  Let $G$ be a graph  of order $n$. Then
  \[
    s_Q(G)\geq \frac{4}{n}\sqrt{nM_{1}(  G)  -4m^{2}}.
  \]
\end{corollary}

\begin{proof}
We  consider (\ref{red}) with $A=Q_{G}$ and $\mathbf{x}=(1/\sqrt{n})\mathbf{e}$ where $\mathbf{e}$ denotes the all ones vector. Then
$\mathbf{x}^{T}A^{2}\mathbf{x }=(1/n) \mathbf{e}^TQ_{G}^{2}\mathbf{e}$.
Let $\mathbf{d}= (d_{1}, d_2, \ldots,d_{n})$ be the vector whose components are the vertex degrees. So
$A_{G}\mathbf{e}=\mathbf{d}$ and
\begin{equation*}
\begin{array}{ll}
 \mathbf{e}^TQ_{G}^{2}\mathbf{e}    & =\mathbf{e}^T(D+A_{G})^{2}\mathbf{e}\\
                                                     &=\mathbf{e}^TD^{2}\mathbf{e}+\mathbf{e}^TA_{G}^{2}\mathbf{e}+\mathbf{e}^{T}A_{G}D\mathbf{e}+\mathbf{e}^TDA_{G}\mathbf{e}\\
                                                    & =M_{1}(  G)  +\left\Vert A_{G}\mathbf{e}\right\Vert ^{2}+(A_{G}\mathbf{e})^T\mathbf{d}+
                                                      (D\mathbf{e})^{T}A_{G}\mathbf{e}  \\
                                                     & =4M_{1}\left(  G\right)  .
\end{array}
\end{equation*}
Thus (\ref{red}) gives
\begin{equation*}
\begin{array}{ll}
   s_{Q}(  G)    & \geq 2\sqrt{4M_1(  G)/n -(4m/n)^{2}}  \\[1.5\smallskipamount]
                     & =(4/n)\sqrt{nM_{1}\left(  G\right)  -4m^{2}}.
\end{array}
\end{equation*}
\end{proof}

\medskip

Next, we apply Theorem \ref{thm:lbd} using the degree vector
$\mathbf{d}=(  d_{1}, d_2, \ldots,d_{n})$. This gives  the following result; it follows directly from (\ref{eq:lbd_reform}).

\begin{corollary}
\label{cor:lbd_degree}
   Let $G$ be a graph. Then
\begin{equation}
 \label{eq:lbd_degree}
   s_Q(G) \ge 2
   \frac{\left(\sum_i d_{i}^{2} \sum_i
\alpha_{i}^{2}  -\left(\sum_id_{i}\alpha_{i}\right)^{2} \right)^{1/2}}{\sum_i
d_{i}^{2}}
\end{equation}
where $\alpha_i=d_i^2 + \sum\limits_{v_{i}v_{j}\in \mathcal{E}} d_j$ for $i \le n$.
\end{corollary}

\medskip
Next, since $G$ is a graph without isolated vertices, we may use (\ref{eq:lbd_reform}) with
\[
  \mathbf{y}=(d_{1}^{-1}, d_{2}^{-1},\ldots, d_{n}^{-1}),
\]
the $n$-tuple of the reciprocal of the
vertex degrees of $G$.

\medskip
\begin{corollary}
\label{Zorro}
Let $G$ be a graph without isolated vertices. Then
\small{
\begin{equation*}
s_{Q}^{2}(G) \geq \frac{4}{\left(\sum_i d_{i}^{-2}\right)^2} \cdot  \left( \sum_i d_{i}^{-2} \sum_j \left(\sum_{v_{j}v_{k} \in \mathcal {E}(G) } d_{k}^{-1}+1\right)^{2}-\left(\sum_i(\sum_{v_{i}v_{k} \in \mathcal{E}(G) }(d_{i}d_{k})^{-1}+d_{i}^{-1}) \right)^{2} \right).
\end{equation*}
}
\end{corollary}
\begin{proof}
The
vector\ $\mathbf{\tau=}\left(  \tau_{1}, \tau_2, \ldots,\tau_{n}\right)
=Q\mathbf{y}$ then satisfies
\begin{equation*}
\tau_{i}=1+
{\textstyle\sum\limits_{v_{i}v_{j}\in \mathcal {E}(G)   }}
d_{j}^{-1}.
\end{equation*}

\bigskip Moreover,

\begin{enumerate}
\item  $\left(
{\textstyle\sum_{i=1}^{n}}
y_{i}^{2}\right)  ^{2}=\left(
{\textstyle\sum\limits_{i=1}^{n}}
\frac{1}{d_{i}^{2}}\right)  ^{2}.$

\item $\left(
{\textstyle\sum\limits_{i=1}^{n}}
\tau_{i}^{2}\right)  =%
{\textstyle\sum\limits_{i=1}^{n}}
\left(  1+
{\textstyle\sum\limits_{v_{i}v_{j} \in \mathcal {E}(G)   }}
\frac{1}{d_{j}}\right)  ^{2}.$

\item $\left(
{\textstyle\sum\limits_{i=1}^{n}}
y_{i}\tau_{i}\right)  ^{2}=\left(
{\textstyle\sum\limits_{i=1}^{n}}
\left(  \frac{1}{d_{i}}+%
{\textstyle\sum\limits_{v_{i}v_{k}\in \mathcal {E}(G)  }}
\frac{1}{d_{k}d_{i}}\right)  \right)  ^{2}.$
\end{enumerate}
Then, considering  $\mathbf{x}= (1/\|\mathbf{y}\|) \mathbf{y}$ in  (\ref{eq:lbd_reform}), the result follows.
\end{proof}

\medskip
We now establish some other lower bounds on $s_Q(G)$ based on other principles.

\begin{theorem} Let $G$ be a graph with vector degrees $\mathbf{d=}\left(  d_{1}, d_2, \ldots,d_{n}\right) $.  Moreover, consider $\mathbf{d}^{\left(
2\right)  }\mathbf{=}\left(  d_{1}^{\left(  2\right)  },d_{2}^{\left(
2\right)  },\ldots,d_{n}^{\left(  2\right)  }\right) $ as the vector of
second degrees of $G$, that is
\[
\mathbf{d}^{\left(  2\right)  }=A\mathbf{d}%
\]
where $A$ is  the adjacency matrix of $G$.\ Then
\begin{equation}\label{pp}
s_{Q}(  G)  \geq\left\vert \frac{
{\textstyle\sum\limits_{i=1}^{n}}
d_{i}^{3}+
{\textstyle\sum\limits_{i=1}^{n}}
d_{i}d_{i}^{\left(  2\right)  }}{M_{1}\left(  G\right)  }-\Upsilon\right\vert
\end{equation}

with
\[
\Upsilon=\min\limits_{\substack{v_{p}v_{q}\in \mathcal {E}(G)   \\d\left(
v_{q}\right)  =\Delta}}\left\{  \frac{\Delta+d_{p}}{2}-\sqrt{\left(
\frac{\Delta+d_{p}}{2}\right)  ^{2}+1-\Delta d_{p}}\right\}  .
\]
\end{theorem}
Note that if $G$ is a bipartite graph then $ \Upsilon=0=q_{n}(G).$

\begin{proof}
In \cite[Theorem $6$]{Mirsky} the following lower bound for the spread $s(
B)  $ of an Hermitian matrix $B=\left(  b_{ij}\right)  $ was shown
\[
s(  B)  \geq\max\limits_{p\neq q}\left\vert \frac{\mathbf{e}%
^{T}B^{3}\mathbf{e}}{\mathbf{e}^{T}B^{2}\mathbf{e}}-\frac{b_{pp}+b_{qq}%
\pm\sqrt{\left(  b_{pp}-b_{qq}\right)  ^{2}+4\left\vert b_{pq}\right\vert
^{2}}}{2}\right\vert .
\]
Replacing $B$ by $Q=Q_G $ one has



$$\mathbf{e}^{T}Q^{3}\mathbf{e}=4\left(
{\textstyle\sum\limits_{i=1}^{n}}
d_{i}^{3}+
{\textstyle\sum\limits_{i=1}^{n}}
d_{i}d_{i}^{\left(  2\right)  }\right).$$









By the Proof of Corollary 29 we get
\[
\mathbf{e}^{T}Q^{2}\mathbf{e}=4M_{1}\left(  G\right)  .
\]
Then

\[
\frac{\mathbf{e}^{T}B^{3}\mathbf{e}}{\mathbf{e}^{T}B^{2}\mathbf{e}}
=\frac{\mathbf{e}^{T}Q^{3}\mathbf{e}}{\mathbf{e}^{T}Q^{2}\mathbf{e}}
=\frac{4\left(
{\textstyle\sum\limits_{i=1}^{n}}
d_{i}^{3}+
{\textstyle\sum\limits_{i=1}^{n}}
d_{i}d_{i}^{\left(  2\right)  }\right)  }{4M_{1}\left(  G\right)  }.
\] Moreover, from the Proof of Theorem $6$ in \cite{Mirsky} one sees that
$$\frac{b_{pp}+b_{qq}\pm\sqrt{\left(  b_{pp}-b_{qq}\right)  ^{2}+4\left\vert
b_{pq}\right\vert ^{2}}}{2}$$
corresponds to the smaller  eigenvalue of the
$2\times2$ submatrix of $B$,
$$
\begin{pmatrix}
b_{pp} & b_{pq}\\
b_{pq} & b_{qq}
\end{pmatrix},
$$
and  we will see that the minimum (for the case of $Q)$
corresponds to the smaller eigenvalue of some $2\times2$ submatrix of $Q$ with
the form $
\begin{pmatrix}
d_{p} & 1\\
1 & d_{q}%
\end{pmatrix}
$. Two cases must be considered.

\begin{enumerate}
\item The submatrix is
$
\begin{pmatrix}
d_{p} & 1\\
1 & d_{q}
\end{pmatrix}
$.
By a straightforward computation, of the mentioned eigenvalue, we obtain

\begin{equation*}
\lambda_{-}=\frac{d_{p}+d_{q}}{2}-\sqrt{\left(  \frac{d_{p}+d_{q}}{2}\right)
^{2}+1-d_{p}d_{q}.}
\end{equation*}

Let $x=\frac{d_{p}+d_{q}}{2}$ and consider the function
$$f\left(x\right)=x-\sqrt{x^{2}+\alpha}, x\in\left(  0,\infty\right),$$
with $\alpha<0$. From the derivative
$f^{\prime}\left(  x\right)  =1-\frac{x}{\sqrt
{x^{2}+\alpha}},$
one easily sees that $f^{\prime}\left(  x\right)<0$, so
$f\left(x\right)$ is strictly decreasing, thus the minimum
$$\Upsilon
=\min\limits_{v_{p}v_{q}\in \mathcal {E}(G)  }\left\{  \frac{d_{p}+d_{q}}
{2}-\sqrt{\left(  \frac{d_{p}+d_{q}}{2}\right)  ^{2}+1-d_{p}d_{q}}\right\}$$

can not be obtained for small degrees. Recall that the maximum vertex degree is
denoted by $\Delta$.
We conclude that
\begin{equation*}
\Upsilon=\min\limits_{\substack{v_{p}v_{q}\in \mathcal {E}(G)   \\d\left(
v_{q}\right)  =\Delta}}\left\{\frac{\Delta+d_{p}}{2}-\sqrt{\left(
\frac{\Delta+d_{p}}{2}\right)  ^{2}+1-\Delta d_{p}}\right\}.
\end{equation*}

\item The submatrix is
$
\begin{pmatrix}
d_{p} & 0\\
0 & d_{q}
\end{pmatrix}
$.

It is clear that its smaller eigenvalue is
$$\min\left\{
d_{p},d_{q}\right\}, $$
thus $\Upsilon=\delta$ is the minimum vertex degree of
$G$.
We recall the above function $f\left( x\right)  =x-\sqrt{x^{2}+\alpha}
$,\ $x\in\left(  0,\infty\right)  $\ with $\alpha<0$.\ If  $x=\delta$,
then $\delta\leq\frac{d_{p}+d_{q}}{2}$, implies
\end{enumerate}
\begin{eqnarray*}
f\left(  \delta\right) & = &\delta-\sqrt{\delta^{2}+\alpha}\geq f\left(  \frac{d_{p}+d_{q}}{2}\right) =
                            \frac{d_{p}+d_{q}}{2}-\sqrt{\left(  \frac{d_{p}+d_{q}}{2}\right)  ^{2}+\alpha}.
\end{eqnarray*}
As the constant $\alpha$ in function $f$ equals the
negative number $\alpha=1-d_{p}d_{q},$ we have
\[
f\left(  \delta\right)  =\delta-\sqrt{\delta^{2}+\alpha}\geq\frac{d_{p}+d_{q}%
}{2}-\sqrt{\left(  \frac{d_{p}+d_{q}}{2}\right)  ^{2}+1-d_{p}d_{q}}.
\]
Moreover, as
\begin{equation*}
\delta\geq\delta-\sqrt{\delta^{2}+\alpha}\geq\frac{d_{p}+d_{q}}{2}%
-\sqrt{\left(  \frac{d_{p}+d_{q}}{2}\right)  ^{2}+1-d_{p}d_{q}},
\end{equation*}
the result follows.
\end{proof}

\begin{remark}
{\rm
If $G=K_{r,s}$, the complete bipartite graph, the lower bound in (\ref{pp}) becomes $[s^2+r^2+s+r]/[s+r].$ \endproof
}
\end{remark}

\begin{theorem} \label{maria}
Let $G$ be a $k$ regular graph. Then
\begin{equation*}
s( G) =s_{Q}( G) \geq k+1.
\end{equation*}
\end{theorem}

\begin{proof} Let $G$ be a $k$-regular graph then,

\begin{enumerate}
\item $A\mathbf{e}=k\mathbf{e}$

\item $\mathbf{d}^{( 2) }=k^{2}\mathbf{e}$

\item $\sum\limits_{i=1}^{n}d_{i}^{3}=nk^{3}$

\item $\sum\limits_{i=1}^{n}d_{i}d_{i}^{( 2)}=nk^{3},$ and

\item $M_{1}( G) =nk^{2}.$

\begin{eqnarray*}
\Upsilon  &=&\min_{\substack{ v_{p}v_{q}\in \mathcal {E}(G)   \\ d(
v_{p}) =\Delta }}\left \{ \frac{k+k}{2}-\sqrt{\left( \frac{k+k}{2}
\right) ^{2}+1-k^{2}} \right \}  \\
&=&k-1.
\end{eqnarray*}

By the inequality in (\ref{pp}), one obtains
\begin{eqnarray}
s_{Q}( G)  &\geq &\left\vert \frac{nk^{3}+nk^{3}}{nk^{2}}-(
k-1) \right\vert   \label{inranreg} \\
&=&\left\vert 2k-\left( k-1\right) \right\vert =k+1.  \notag
\end{eqnarray}
Thus the statement follows.
\end{enumerate}
\end{proof}

\begin{remark}
{\rm
For $k>3$ the previous lower bound improves the lower bound given
in Corollary \ref{3.9}.
} \endproof
\end{remark}


\section{Upper bounds}

\medskip

In  \cite{Das},  using the Mirsky's upper bound mentioned above, it was shown
that for a graph $G$ with $n\geq5$ vertices and $m\geq1$ edges, the following
inequality holds
\[
   s_{L}\left(  G\right)  \leq\sqrt{2M_{1}(  G)  +4m-\frac{8m^{2}%
}{n-1}}.
\]
Here  equality holds if and only if $G$ is one of the graphs $K_{n},\ G(
\frac{n}{4},\frac{n}{4})  ,\ K_{1}\vee2K_{\frac{n-1}{2}},\ \overline
{K}_{\frac{n}{3}}\vee2K_{\frac{n}{3}},K_{1}\cup K_{\frac{n-1}{2},\frac{n-1}%
{2}},K_{\frac{n}{3}}\cup K_{\frac{n}{3},\frac{n}{3}}.\ $The graph $G(
r,s)  $ is the graph obtained by joining each vertex of the subgraph
$\overline{K}_{s}$ of $K_{r}\vee\overline{K}_{s}$ to all the vertices of
$\overline{K}_{s}$ of another copy of $K_{r}\vee\overline{K}_{s}$. Here $G \vee G'$ is the usual join operation between two graphs $G$ and $G'$.

\begin{theorem}
\label{MK"}Let $G$ be an $\left(  n,m\right)$-graph. Then
\begin{equation}
 \label{eq:sQub}
   s_{Q}(  G)  \leq\sqrt{2\left(
{\displaystyle\sum\limits_{i=1}^{n}}
d_{i}^{2}+2m\right)  -\frac{8m^{2}}{n}}=\sqrt{2M_{1}\left(  G\right)
+4m-\frac{8m^{2}}{n}}.
\end{equation}
The equality is attained if and only if $G\simeq K_{\frac{n}{2},\frac{n}{2}}.$
\bigskip
\end{theorem}

\begin{proof}
Since $Q=Q(  G)  $ is a normal matrix, by applying Theorem
\ref{MK1} to $Q$ we obtain
\[
s_{Q}(  G)  =s(  Q)  \leq\sqrt{2\left\Vert Q\right\Vert
_{F}^{2}-\frac{2}{n}\left(  {\rm tr}\,Q\right)  ^{2}}
\]
with equality if and only if $\ $the eigenvalues $q_{1}, q_2, \ldots,q_{n}$
satisfying the following condition
\[
  (*) \;\;\; q_{2}=q_3= \cdots=q_{n-1}=\frac{q_{1}+q_{n}}{2}.
\]
As $\left\Vert Q\right\Vert _{F}^{2}=M_{1}(  G)  +2m$ and
${\rm tr}\, Q=2m$, the result follows.
If  condition $(*)$ holds then
\[
{\rm tr}\, Q  =nq_{2}
\]
so
\[
q_{2}=\frac{2m}{n}=\frac{1}{2}\frac{\mathbf{e}^{T}Q\left(  G\right)
\mathbf{e}}{\mathbf{e}^{T}\mathbf{e}}\leq\frac{1}{2}q_{1},
\]
by  \cite[Lemma 1.1]{Lingsheng}.
Then
\[
q_{1}+q_{n}=2q_{2}\leq q_{1}
\]
so $q_n=0$ and $q_{2}=\frac{1}{2}q_{1}$. This gives
\[
q_{1}=2q_{2}=\frac{4m}{n}. 
\]
Thus, $G$ is a regular bipartite graph and the statement holds. Conversely,
if $G\simeq K_{\frac{n}{2},\frac{n}{2}}$, by a standard verification, the
inequality in Theorem \ref{MK"} holds with equality.
\end{proof}

\begin{corollary}
Let $G_{k}$ be a $k$-regular graph with $n$ vertices. Then
\[
s_{Q}(  G)  \leq \sqrt{2nk}.
\]
Here equality is attained if and only if $G\simeq K_{\frac{n}{2},\frac{n}{2}}.$
\end{corollary}

\begin{proof}
We get $M_{1}(  G_{k})  =nk^{2}$ and
$m=\frac{nk}{2}$. Thus,
\[
  2M_{1}(  G_{k})  +4m-\frac{8m^{2}}%
{n}=2nk^{2}+4\frac{nk}{2}-\frac{8}{n}\left(  \frac{nk}{2}\right)  ^{2}%
=2nk^{2}+2nk-2nk^{2}=2nk.
\]
By Theorem \ref{MK"}, the result now follows.
If\ $G\simeq K_{\frac{n}{2},\frac{n}{2}}$, then $G$ is a regular bipartite
graph with $k=\frac{n}{2}$, so $\sqrt{2nk}=\sqrt{2n\frac{n}{2}}=n=\mu
_{1}(  K_{\frac{n}{2},\frac{n}{2}})  =s_{Q}(  K_{\frac{n}%
{2},\frac{n}{2}})  .$
\end{proof}

\begin{corollary}
Let $G$ be an $(  n,m)  $-graph. Then
\begin{equation}
s_{Q}(  G)  \leq\sqrt{2m(  \frac{2m}{n-1}+\frac{n-2}%
{n-1}\Delta+(  \Delta-\delta) (  1-\frac{\Delta}{n-1}))  +4m-\frac{8m^{2}}{n}}.\label{ineq2}
\end{equation}
The equality is attained if and only if $G\simeq K_{\frac{n}{2},\frac{n}{2}}.$
\end{corollary}

\begin{proof}
In \cite{Das_Zagreb} it was shown that
\begin{equation}
 \label{upper_Zagreb}
    M_{1}(  G)  \leq m(  \frac{2m}{n-1}+\frac{n-2}{n-1}%
\Delta+(  \Delta-\delta)  (  1-\frac{\Delta}{n-1})
)
\end{equation}
with equality if and only if $G$ is either a star, a regular graph or a
complete graph $K_{\Delta+1}$ with $n-\Delta-1$ isolated vertices.
Replacing $M_{1}(  G)  $ in (\ref{eq:sQub}) by its upper bound in
(\ref{upper_Zagreb}) the result follows. Equality holds in (\ref{ineq2}) if and only
if equality holds in both (\ref{eq:sQub}) and (\ref{upper_Zagreb}), or equivalently
$G\simeq K_{\frac{n}{2},\frac{n}{2}}.$
\end{proof}

\section{Comparison of bounds}

This section deals with a comparison of  some of the bounds  presented in this work.
We firstly compare the bound in Theorem \ref{barnes_3} with the lower bound for $s_Q(G)$ (depending on same parameters) found in \cite[Corollary 2.3]{Liu2}:

\[
  s_Q(G) \ge  \frac{1}{n-1} \left( (n \Delta)^2 + 8(m-\Delta)(2m-n\Delta) \right)^{\frac{1}{2}}.
\]

Let $L_1(G)$ and $L_2(G)$ denote the bound from Theorem \ref{barnes_3} and \cite{Liu2}, respectively, so
\[
\begin{array}{l}
   L_1(G)  =\left(  \left(  \Delta-\delta\right)  ^{2}+2\Delta+2\delta+4\right)^{\frac{1}{2}}, \\*[\smallskipamount]
   L_2(G)  =\frac{1}{n-1} \left( (n \Delta)^2 + 8(m-\Delta)(2m-n\Delta) \right)^{\frac{1}{2}}.
\end{array}
\]
Observe that $L_1(G)$ only depends on the minimum and maximum degrees, not $n$ and $m$.
Let $\bar{d}=(1/n)\sum_{i=1}^n d_i=2m/n$ denote the average degree in $G$. So
\[
\begin{array}{ll}
   L_2(G)
       &=\frac{n}{n-1} \left( \Delta^2 + \frac{8(m-\Delta)(2m-n\Delta)}{n^2} \right)^{\frac{1}{2}} \\
       &=\frac{n}{n-1} \left( \Delta^2 + (4\bar{d}-\frac{8\Delta}{n})(\bar{d}-\Delta) \right)^{\frac{1}{2}} \\
\end{array}
\]
which shows that $L_2(G)$ is determined by the maximum and average degree as well as $n$. Here
$\bar{d}-\Delta \le 0$, and $\bar{d}-\Delta = 0$ precisely when  $G$ is regular.

The next result relates the two lower bounds as a function of certain graph properties.

\begin{theorem}
  \label{thm:L1andL2}
  Let $G$ be a an $(n,m)$-graph with $n>2$.
\begin{description}
\item  $(i)$ Assume $G$ is a $k$-regular graph. Then $L_1(G)=2\sqrt{k+1}$ and $L_2(G)=\frac{n}{n-1}k$.
  Therefore,  $L_2(G) > L_1(G)$ except when  $k \le 3$ $($and $n$ arbitrary$)$ or $k=4$ and $n \ge 10$.

\item  $(ii)$ Assume $G$ is connected and contains a pendant vertex. Then
  $L_2(G) \le  \frac{n}{n-1}\Delta$ and $L_1(G)=\sqrt{\Delta^2+7}$. In particular, $L_2(G)<L_1(G)$ holds if
  $\frac{2n-1}{(n-1)^2} \Delta^2 < 7$.
\end{description}
\end{theorem}
\begin{proof}
 (i) The two expressions follow from the calculation above as  $\bar{d}=\delta=\Delta=k$. Consider the case when $G$ is regular, say of degree $k$. Then  $L_2(G)\le L_1(G)$ gives $\frac{n}{n-1}k \le 2\sqrt{k+1}$, or $1-\frac{1}{n} \ge \frac{k}{2\sqrt{k+1}}$. Here the right hand side is greater than $1$ precisely when $k \ge 5$, and the conclusion then follows.

 (ii) Since there is a pendant vertex, $\delta=1$. This gives
 \[
     L_1(G)=\sqrt{(\Delta-1)^2+2\Delta+2+4}=\sqrt{\Delta^2+7}.
 \]
 We have $\Delta>\delta=1$, for if $\Delta=1$, $G$ would be a perfect matching, contradicting that $G$ is connected and $n>2$. Therefore $\bar{d}-\Delta < 0$. Moreover, as $G$ is connected, $m \ge n-1 \ge \Delta$. So $m\ge \Delta$, and using that $2m=\sum_i d_i$, we easily derive $4\bar{d}\ge 8\Delta/n$. Therefore
 \[
   L_2(G)
       =\frac{n}{n-1} \left( \Delta^2 + (4\bar{d}-\frac{8\Delta}{n})(\bar{d}-\Delta) \right)^{\frac{1}{2}} \\*[\smallskipamount]
       \le \frac{n}{n-1} \left( \Delta^2 \right)^{\frac{1}{2}} \\*[\smallskipamount]
       =\frac{n}{n-1}\Delta.
\]
Therefore, if $\frac{n}{n-1}\Delta < \sqrt{\Delta^2+7}$, then $L_2(G)<L_1(G)$. The last statement follows from this.

\end{proof}

Note that the lower bound $L_1(G)$ for the regular case is worse than the lower bound in Theorem \ref{maria}.

Consider again our lower bound on the signless Laplacian spread $s_Q(G)$
\begin{equation}
 \label{eq:bd}
   \eta(G):=2\max_{\mathbf{x} \in B_n}\|Q\mathbf{x} -(\mathbf{x} ^TQ\mathbf{x} )\mathbf{x} \|.
 \end{equation}
In the  proof we obtained the bound by some calculations in which a single inequality was involved, namely when we used that ``minmax'' is at least as large as ``maxmin'', for the function involved. Unfortunately, we cannot show that equality holds here. The reason for this is basically that $\|(Q-tI_n)\mathbf{x} \|$ is not a concave function of $\mathbf{x} $ and, also, $B_n$ is not a convex set, so general minmax theorems may not be applied to our situation.

However, it is interesting to explore further the quality of the best bound one gets from the minmax principle.
To do so, consider the function
\[
       f(\mathbf{x} ) = 2\|Q\mathbf{x} -(\mathbf{x} ^TQ\mathbf{x} )\mathbf{x} \|
\]
so that $\eta(G)=\max_{\mathbf{x}  \in B_n} f(\mathbf{x} )$.  Note that $f$ is a complicated function, obtained from a multivariate polynomial of degree six (by taking the square root, although that can be removed for the maximization). We  consider an extremely simple approach to approximately maximize $f$ over the unit ball;  we perform a few iterations $K$ of the following gradient method with a step length $s>0$:

\medskip
\noindent {\bf Algorithm: Simple gradient search.}
\vspace{-0.1cm}
\begin{tabbing}
1. Let $\mathbf{x} =(1/\sqrt{n})\mathbf{e}$, and $\eta=f(\mathbf{x} )$. \\
2. for \=$k=1, 2, \ldots, K$ \\
3. Output $\eta$.
\end{tabbing}
In each iteration, we make a step in the direction of the (numerical) gradient, even if the new function value could be less. Thus we avoid line search. The disadvantage is that we may not approximate a local maximum so well, but the advantage is that we can escape a local maximum and go towards  another with higher function value. The procedure is very simple, and heuristic, and we typically only perform a few iterations $K$ (around 10 or 20). We have used constant step length $s$, but also variable step length (being a decreasing function of the iteration number).

In the table below we give some computational results, for 5 random, connected graphs, showing all previous lower bounds we have discussed and the new bound $\eta$.  The notation in the table is the following:


\begin{eqnarray*}
 liu_{2.2}  & =  &  \cite{Liu2}, \text{Theorem}\,  2.2\\
liu_{2.3}   & =  &  \cite{Liu2}, \text{Corollary} \, 2.3\\
meg_1       & =  & \text {bound in Theorem}  \, \ref{O1} \\
meg_2       & =  &  \text{bound in Theorem} \,  \ref{barnes_3} \\
Ncon        & =  &  \text{bound in Corollary}   \, \ref {cor:lbd}\\
Z1    & =  &  \text{from Theorem \ref{thm:lbd} (minmax principle), using inverse of degrees} \\
Z2    & =  &  \text{from Theorem \ref{thm:lbd}, using vector of $d_i^{-3}$} \\
\eta       & =  & \text{best bound from simple gradient method for the function $\eta(\mathbf{x})$, 10 iterations} \\
spread      & =  & \text{exact}\,  s_Q \, \text{spread}.
\end{eqnarray*}

\medskip

{\scriptsize
\begin{tabular}{|c|c|c|c|c|c|c|c|c|c|c|c|c|} \hline

 $n$  & $m$  & $\Delta$&$\delta$& $liu_{2.2}$&$liu_{2.3}$& $meg_1$&  $meg_2$ & $ Ncon$& $Z1$ & $Z2$ & $\eta$ &  $spread$ \\ \hline
  40  & 634 &  36 &  27  &    32.60  &  28.68  & 11.91 &  14.53  &  7.76  & 11.76 &  19.31 &     38.39  &   39.19 \\ \hline
 40  & 519  & 32 &  20  &    26.99 &   21.38  & 11.81  & 15.87  & 11.96  & 17.75  & 26.68   &   31.07  &   36.03  \\ \hline
 40  & 322 &  23  &  9  &    17.07  &  11.06  &  9.97  & 16.25  & 11.83  & 17.79 &  23.16  &    25.14   &  26.34 \\ \hline
 40  & 273  & 19  &  9   &   14.42  &   9.69  &  8.98  & 12.65 &  10.22  & 14.87  & 18.41   &   20.48   &  22.50  \\ \hline
 40  & 346  & 22  & 12   &   18.01  &  13.74  &  9.66  & 13.11  &  9.81  & 15.00 &  21.82   &   23.94   &  26.33 \\ \hline
\end{tabular}
}

For the last example above we next show the value of $\eta$ during the 10 iteration of the gradient search algorithm, and we see that that maximum, in this case,  was  found in iteration 4:

\medskip
{\scriptsize
\begin{tabular}{|c|c|c|c|c|c|c|c|c|c|c|c|} \hline
   \text{iteration} & 1 & 2 & 3 & 4 & 5 & 6 & 7 & 8 & 9 & 10  \\\hline
    $f(\mathbf{x})$
    &9.80
    & 21.77
   & 23.41
   & 23.94
   & 22.81
   & 18.77
   & 22.34
   & 17.32
   & 23.20
   & 19.34 \\\hline
\end{tabular}
}

\medskip
These, and similar, experiments  clearly show that $\eta(G)$ is a {\em very good} lower bound on the signless Laplacian spread $s_Q(G)$. Although the exact computation of $\eta(G)$ may be hard, we see that a simple gradient algorithm finds very good approximations, and lower bounds on $s_Q(G)$, in a few iterations. Of, course, the result of such an algorithm is not an analytical bound in terms of natural graph parameters. But every bound needs to be computed, and, in practice,  its computational effort should always be compared to the work of using an eigenvalue algorithm for computing the largest and smallest eigenvalue of $Q$, and finding $s_Q(G)$ in that way.

Finally, we remark that it is  possible to use the results above to find such an analytical bound which is quite good: compute the exact gradient (of $f(\mathbf{x})^2$) at the  constant vector and make one iteration in the gradient algorithm; let $\hat{\mathbf{x}}$ be the obtained unit vector, and compute the bound $f(\hat{\mathbf{x}})$. We leave this computation to the interested reader.

\bigskip\bigskip\bigskip

%
%
%
%
%


\textbf{Acknowledgments}. Enide Andrade was supported in part by the
Portuguese Foundation for Science and Technology (FCT-Funda\c{c}\~{a}o para a
Ci\^{e}ncia e a Tecnologia), through CIDMA - Center for Research and
Development in Mathematics and Applications, within project
UID/MAT/04106/2013. M. Robbiano was partially supported by project Proyecto VRIDT UCN16115.


\begin{thebibliography}{9}
\bibitem {Nair}N. Abreu. Old and new results on algebraic connectivity of
graphs, Lin. Algebra Appl. 423 (2007): 53--73.

\bibitem {Enide15} E. Andrade, D. M. Cardoso, M. Robbiano, J. Rodr\'{i}guez.
Laplacian spread of graphs:  Lower bounds and relations with invariant parameters,  Lin. Algebra Appl. 486 (2015): 494--503.


\bibitem {Barnes-Hoffman}E.R. Barnes, A.J. Hoffman. Bounds for the spectrum of
normal matrices, Lin. Algebra Appl. 201 (1994): 79-90.

\bibitem{BloomfieldWatson} P. Bloomfield, G. S. Watson. The ineficiency of least squares, Biometrika 62 (1975): 121--128.

\bibitem {BrualdiRyser91} R.A.~Brualdi, H.J.~Ryser. {\em Combinatorial Matrix Theory}, Cambridge University Press, Cambridge, 1991.

\bibitem {Das}X. Chen, K. Ch. Das. Some Results on the Laplacian Spread of a
graph, Lin. Algebra Appl. 505 (2016): 245-260.

\bibitem {Das_Zagreb}K. Ch. Das. Maximizing the sum of the squares of the
degrees of a graph, Discrete Mathematics. 285 (2004): 57-66.


\bibitem {Das2}K. Ch. Das, S. A. Mojallal. Relation between signless Laplacian
energy, energy of graph and its line graph, Lin. Algebra Appl. 493 (2016): 91--107.

\bibitem {lapl4}D. Cvetkovi\'{c}, P. Rowlinson, S. Simi\'{c}. Signless
Laplacians of finite graphs, Lin. Algebra Appl. 423 (2007): 155--171.

\bibitem{Choi} Hyeong-Ah Choi, Kazuo Nakajima, Chong S. Rim,
Graph bipartization and via minimization, SIAM J. Disc. Math. 1,(1989): 38--47.

\bibitem {Fan-Fallat}S. Fallat, Yi-Zheng Fan. Bipartiteness and the least
eigenvalue of signless Laplacian of graphs, Lin. Algebra Appl. 436 (2012): 3254-3267.


\bibitem {XDZ2}L. Feng, Q. Li, X. D. Zhang. Some sharp upper bounds on the
spectral radius, Taiwanese J. Math. 11 (2007): 989-997.



\bibitem {Fiedler}M. Fiedler. Algebraic connectivity of graphs, Czech. Math. J 23 (1973): 298-305.

\bibitem{Garey} M. R. Garey, D. S. Johnson. A guide to the theory of NP- completeness. Computers and intractability, 1979.

\bibitem {G}D. A. Gregory, D. Heshkowitz, S. J.  Kirkland. The spread of
the spectrum of a graph, Linear Algebra Appl. 332-334 (2001): 23-35.


\bibitem {lapl1}R. Grone, R. Merris. The Laplacian spectrum of a graph II,
SIAM J. Discr. Math. 7 (1994): 221--229.

\bibitem {lapl2}R. Grone, R. Merris, V. S. Sunder. The Laplacian spectrum of a
graph, SIAM J. Matrix Anal. Appl. 11 (1990): 218--238.

\bibitem{GD} I. Gutman, K. C. Das. The first Zagreb index 30 years after, MATCH Commun. Math. Comput. Chem. 50 (2004): 83-92.

\bibitem {inci}I. Gutman, D. Kiani, M. Mirzakhah, B. Zhou. On incidence energy
of a graph, Lin. Algebra Appl. 431 (2009): 1223--1233.


\bibitem{Jensen} S. T. Jensen. The Laguerre- Samuelson inequality with extensions and applications in statistics and matrix theory. Master's thesis, McGill University, Montereal, 1999.
    
\bibitem{Jia} Z- Jia. An extension of Styan's inequality, Gongcheng Shuxue Xuebao 13 (1996): 122--126 (Chinese, English summary).

\bibitem {Jiang-Zhan}E. Jiang, X. Zhan. Lower bounds for the Spread of a
Hermitian Matrix, Lin. Algebra and Appl. 256 (1997): 153-163.

\bibitem {Jonhson et al}C. R. Johnson, R. Kumar, H. Wolkowicz. Lower bounds
for the spread of a matrix, Linear Algebra Appl. 71 (1985): 161-173.

\bibitem{Karp} R. M. Karp. Reducibility among combinatorial problem. in: R.E. Miller, J. W. Thatcher (eds) Complexity of computer computations. Plenum, New York, 85-104.

\bibitem {kirkland_et_al_02}S. J. Kirkland, J. J. Molitierno, M. Neumann, B.
L. Shader. On graphs with equal algebraic and vertex connectivity, Lin. Algebra Appl. 341 (2002): 45--56.

\bibitem {Liu2}M. Liu, B. Liu. The signless Laplacian spread, Linear
Algebra and Appl. 432 (2010): 505-514.

\bibitem{MerikoskiKumar} J. K. Merikoski, R. Kumar. Characterizations and lower bounds for the spread of a normal matrix, Lin. Algebra Appl. 364 (2003): 13--31.

\bibitem {Merris}R. Merris. Laplacian matrices of a graph: A survey, Lin.
Algebra Appl. 197\&198 (1994): 143-176.

\bibitem {Mirsky}L. Mirsky. The spread of a matrix, Mathematika 3 (1956): 127-130.

\bibitem {Pecaric}D. S Mitrinovi\'{c}, I. E. Pe\v{c}ari\'{c}, A. M.
Fink.\ Classical and New Inequalities in Analysis. {\em Mathematics and its Applications} Springer Science + Business
Media Dordrecht B. V (1993).

\bibitem {Nylen}P. Nylen, T. Y. Tam. On the spread of a Hermitian matrix and a
conjecture of Thompson, Linear and Multilinear Algebra 37 (1994): 3-11.

\bibitem {Carla}C.S. Oliveira, L.  Silva de Lima, Nair Maria Maia de Abreu,
Steve Kirkland. Bounds on the Q-spread of a graph, Lin. Algebra and Appl.
432 (2010): 2342-2351.

\bibitem{Raman} V. Raman, S. Saurabh, S. Sikdar.
Improved Exact Exponential Algorithms for Vertex Bipartization and Other Problems. Lecture Notes in Computer Science (3701) 375--389, Springer-Verlag Berlin Heidelberg, 2005.

\bibitem{Reed} B. Reed, K. Smith and A. Vetta. Finding Odd Cycle Transversals, Operations Research Letters. 32, (2004): 299--301.

\bibitem {Lingsheng}L. Shi. Bounds on the Laplacian spectral radius, Lin. Algebra Appl. 422 (2007): 755-770.

\bibitem {Steele}J.M. Steele. {\em The Cauchy-Schwarz Master Class}. Cambridge University Press, Cambridge, 2004.

\bibitem{Styan} G. P. H. Styan. On some inequalities associated with ordinary least squares and the Kantorovich inequality, Acta Univ. Tamper. Ser. A 153 (1983): 158--166.

\bibitem{Yannakakis}
M. Yannakakis. Node- and Edge-deletion NP-complete Problems, Proc. 10th Annual ACM Symp. on Theory of Computing, San Diego, CA, October 1978,  253--264.


\bibitem {Liu}Z. You, B. Liu, Guangzhou. The Laplacian spread of graphs,
Czech Math J. 62 (1) (2012): 155-168.

\bibitem {XDZ}X. D. Zhang, R. Luo. The spectral radius of triangle free
graphs, Australas. J. Combin. 26 (2002): 33-39.

\bibitem {Zheng et al}Y. Zheng Fan, X. Jing, Y. Wang, D. Liang. The Laplacian
spread of a tree, Discrete Mathematics and Theoretical Computer Science, DMTCS, 10:1
(2008): 79-86.

\bibitem {Yi-Zhenf and Fallat}Y. Zheng Fan, S. Fallat. Edge bipartiteness and
signless Laplacian spread of graphs, Appl. Anal. Discrete Math. 6 (2012): 31-45.
\end{thebibliography}
\end{document}